\documentclass[a4paper,12pt]{amsart}
\usepackage{amsfonts,amsmath,amssymb,amsthm,mathtools}
\usepackage{bbm}
\usepackage{geometry} 
\usepackage[backref=page]{hyperref} 
\usepackage{hyperref}
\usepackage{fixmath} 
\usepackage{paralist} 
\usepackage{verbatim} 

\usepackage{graphicx}

\hypersetup{
pdftitle={critical case},
pdfsubject={Mathematics},
pdfauthor={Elisabeth Bauernschubert},
pdfkeywords={Excited random walk in a random environment, cookies of strength 1, recurrence, transience, critical branching process in a random environment with immigration.}
}

\theoremstyle{plain}
\newtheorem{theorem}{Theorem}[section]

\newtheorem{lemma}[theorem]{Lemma}

\theoremstyle{remark}
\newtheorem{remark}[theorem]{Remark}
\newtheorem{example}[theorem]{Example}
\theoremstyle{definition}
\newtheorem{definition}[theorem]{Definition}
\newtheorem{assumptionletter}{Assumption}

\newcommand{\N}{\mathbb{N}}
\newcommand{\Z}{\mathbb{Z}}

\newcommand{\bpire}{$(Z_n)_{n\geq0}$ }
\newcommand{\arp}{$(X_n)_{n\geq0}$ }

\DeclareMathOperator{\Var}{Var}

\begin{document}

\title[Recurrence/transience of critical BPIRE, application to ERW]{Recurrence and transience of critical branching processes in random environment with immigration and an application to excited random walks}

\author{Elisabeth Bauernschubert}
\address{Mathematisches Institut\\
    Eberhard Karls Universit\"at T\"ubingen\\
    Auf der Morgenstelle 10\\
    72076 T\"ubingen, Germany}
\email{elisabeth.bauernschubert@uni-tuebingen.de}

\date{January 16, 2013}

\begin{abstract}
We establish recurrence and transience criteria for critical branching processes in random environment with immigration. These results are then applied to discuss recurrence and transience of a recurrent random walk in a random environment on $\Z$ that will be disturbed by cookies inducing a drift to the right of strength 1.
\end{abstract}

\maketitle

\section{Introduction}

This article complements \cite{Bauernschubert12}. In \cite{Bauernschubert12} the author considers a random walk in random environment on $\Z$ which is transient to the left and that will be disturbed by cookies of strength 1 to the right. As can be seen in \cite{Bauernschubert12}, the study of special kinds of branching processes is essential to obtain results on recurrence and transience of these excited random walks.

Therefore, let us first motivate the discussion of critical branching processes in random environment with immigration (critical BPIRE for short) within the present article by introducing the random walk we are dealing with.

\subsection{Excited random walk in random environment} 
Our model is explained as follows.
Consider a sequence $(p_x)_{x\in\Z}\in (0,1)^{\Z}$ and put a random number $M_x$ of cookies on every integer $x\in\Z$. Now a nearest neighbor random walk $(S_n)_{n\geq0}$ is started at 0 with the following transition probabilities.
If the random walker comes to site $x$ and if there is still at least one cookie on this site, he removes one cookie and jumps to $x+1$. Otherwise he makes a step to the right with probability $p_x$ and to the left with probability $1-p_x$.
For an illustration of this model see Figure~\ref{fig:model}, previously been presented in \cite{Bauernschubert12}.

\begin{figure}[!ht]
  \centering
  \scalebox{0.65}{\includegraphics[width=\columnwidth]{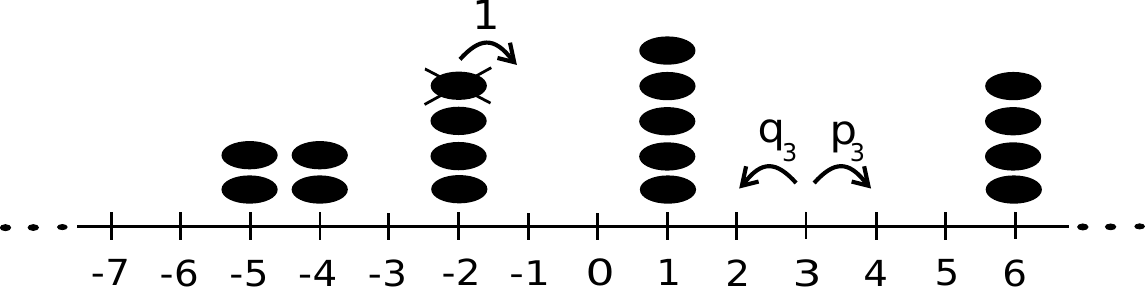}}
  \caption{Model of the random walk. (see \cite{Bauernschubert12})}
  {\parbox{0.85\textwidth}{\small If there are cookies at his current position $x\in\Z$, the random walker removes one and makes a step to $x+1$. If there is no cookie he jumps to the right with probability $p_x$ and to the left with probability $q_x:=1-p_x$.}}
  \label{fig:model}
\end{figure}

The cookies in our model have maximal strength and induce a drift to the right. On the other hand, we will assume a random environment $(p_x)_{x\in\Z}$ that makes a random walk in random environment (RWRE for short), i.e.\ a random walk where $M_x=0$ for all $x\in\Z$, be recurrent. So the question arises when the drift caused by the cookies succeeds in forcing the random walk to $+\infty$. In Theorem \ref{th:1}, criteria for transience and recurrence of the process are given.

First, let us introduce the notation for the model. Set $\Omega:=([0,1]^{\N})^{\Z}$. The elements from $\Omega$ are chosen at random according to a probability measure $\mathbb{P}$ on $\Omega$ with corresponding expectation operator $\mathbb{E}$. For fixed environment $\omega=((\omega(x,i))_{i\geq 1})_{x\in\Z}\in\Omega$ and $z\in\Z$ define a nearest-neighbor random walk $(S_n)_{n\geq0}$ on a suitable probability space $\Omega'$ with probability measure $P_{z,\omega}$, which satisfies
\begin{align*}
    P_{z,\omega}[S_0=z]&=1,\\
    P_{z,\omega}[S_{n+1}=S_n+1|(S_m)_{1\leq m\leq n}]&= \omega(S_n,\#\{m\leq n:\: S_m=S_n\}),\\
    P_{z,\omega}[S_{n+1}=S_n-1|(S_m)_{1\leq m\leq n}]&= 1-\omega(S_n,\#\{m\leq n:\: S_m=S_n\}).
\end{align*}
The value of $\omega(x,i)$ serves as the transition probability from $x$ to $x+1$ upon the $i$-th visit at site $x$. Furthermore, define $P_z[\cdot]:=\mathbb{E}[P_{z,\omega}[\cdot]]$ as the \emph{annealed} or \emph{averaged} probability measure with corresponding expectation operator $E_z$. The random walk $(S_n)_{n\geq0}$ is called \emph{recurrent} (\emph{transient}) if $S_n=0$ infinitely often ($\lim_{n\to\infty}S_n\in\{\pm \infty\}$) $P_0$-a.s.

With the convention $\sup\emptyset =0$, the number of cookies of strength 1 at site $x\in\Z$ is defined by
\begin{align*}
    &M_x:=\sup\{i\geq 1:  \; \omega(x,j)=1\;\text{ for all } 1\leq j\leq i\}
\end{align*}

In this article we postulate the following for the model.
\begin{assumptionletter}
         \label{as_A}
         \renewcommand{\theenumi}{A.\arabic{enumi}}
         \renewcommand{\labelenumi}{\theenumi}
      There is $\mathbb{P}$-a.s.\ $(p_x)_{x\in\Z}\in (0,1)^{\Z}$ such that:
         \begin{enumerate}
                 \item \label{as_A1} It holds $\mathbb{P}$-a.s.\ that
                 $\omega(x,i)=p_x$ for all $i> M_x$. Furthermore, $\mathbb{P}[p_x=\frac{1}{2}]<1$.
                 \item \label{as_A2} $(p_x,M_x)_{x\in\Z}$ is i.i.d.
		 \item \label{as_A3} $\mathbb{E}[|\log \rho_0|]<\infty$ and $\mathbb{E}[\log\rho_0]=0$ where $\rho_x:=({1-p_x})p_x^{-1}$ for $x\in\Z$.
		 \item \label{as_A4} $\mathbb{P}[M_0=\infty]=0$ and $\mathbb{P}[M_0=0]>0$.
         \end{enumerate}
\end{assumptionletter} 

If
$M_x=0$ $\mathbb{P}$-a.s.\ for all $x\in\Z$, assumptions \ref{as_A2} and \ref{as_A3} imply that the RWRE is recurrent, i.e.\ $-\infty=\liminf_{n\to\infty}S_n < \limsup_{n\to\infty}S_n= \infty$ ${P_0}$-a.s. For detailed results about RWRE see e.g.\ \cite{Solomon75}.

Under Assumption \ref{as_A}, $(S_n)_{n\geq0}$ can be seen as a recurrent random walk in random environment disturbed by cookies of strength 1 to the right.
In accordance to RWRE and excited random walk (ERW), our model is called \emph{excited random walk in random environment} (ERWRE for short).
In Section \ref{sec:ERWRE} we show the following recurrence and transience criteria for the ERWRE.

\begin{theorem}
    \label{th:1}
    Let Assumption \ref{as_A} hold and assume that $\mathbb{E}[|\log \rho_0|^{\delta}]<\infty$ for every $0<\delta<6$.
      \begin{itemize}
        \item[(i)] If $\mathbb{E}[(\log_+ M_0)^{2+\epsilon}]<\infty$ for some $\epsilon>0$, then   $S_n=0$ infinitely often $P_0$-a.s.
        \item[(ii)] If $\liminf_{t\to\infty}(t^{\lambda}\cdot \mathbb{P}[\log{M_0}>t])> 0$ for some $0<\lambda<2$, then $\lim_{n\to\infty}S_n=+\infty$ $P_0$-a.s. 
      \end{itemize}
\end{theorem}

\begin{remark} The tail assumption $\liminf_{t\to\infty}(t^{\lambda}\cdot \mathbb{P}[\log{M_0}>t])> 0$ implies that
          $\mathbb{E}[(\log_+ M_0)^{2-\epsilon}]=\infty$ for some $\epsilon>0$.
\end{remark}

\begin{remark}
  In the classical model of the ERW the underlying process is a simple symmetric random walk. Criteria for the recurrence and transience behavior of the classical ERW are given in Theorem 3.10 in \cite{KosyginaZerner12}. Hence, if Assumption \ref{as_A} holds with $\mathbb{P}[p_x=\frac{1}{2}]=1$ in \ref{as_A1}, the process $(S_n)_{n\geq0}$ ist recurrent if and only if $\mathbb{E}[M_0]\leq 1$. Note that this criterion is different to the one in Theorem \ref{th:1}.
\end{remark}

Theorem \ref{th:1} should be compared to the following result from \cite{Bauernschubert12}, where the underlying random walk in random environment is transient to the left and where the following recurrence and transience criteria for the ERWRE were obtained.
\begin{theorem}[\cite{Bauernschubert12}]
\label{th:transientERWRE}
Let assumptions \ref{as_A1}, \ref{as_A2} and \ref{as_A4}
hold and assume that $\{p_x, M_x, x\in\Z\}$ is independent under $\mathbb{P}$, $\mathbb{E}[|\log \rho_0|]<\infty$, $\mathbb{E}[\log\rho_0]>0$
and $\mathbb{E}[p_0^{-1}]<\infty$.
\begin{itemize}
 \item[(i)] If $\mathbb{E}[\log_+ M_0]<\infty$, then $\lim_{n\to\infty}S_n=-\infty$ $P_0$-a.s.
 \item[(ii)]  If $\mathbb{E}[\log_+ M_0]=\infty$ and if $\limsup_{t\to\infty}(t\cdot\mathbb{P}[\log{M_0}>t])<\mathbb{E}[\log \rho_0]$, then $S_n=0$ infinitely often $P_0$-a.s.
 \item[(iii)] If
$\liminf_{t\to\infty}(t\cdot\mathbb{P}[\log{M_0}>t])>\mathbb{E}[\log \rho_0]$, then $\lim_{n\to\infty}S_n=+\infty$ $P_0$-a.s. 
\end{itemize}
\end{theorem}

Excited random walks go back to Benjamini and Wilson \cite{Benjamini} and have been further studied and extended among others by Zerner in \cite{Zerner05, Zerner06}, by Basdevant and Singh in \cite{Basdevant08b, Basdevant08a} and by Kosygina and Zerner in \cite{Zerner08}. A survey on ERW is given by Kosygina and Zerner in \cite{KosyginaZerner12}.
The novelty in our model are the random transition probabilities on sites without cookies and the unbounded number of cookies per site. However, we consider only cookies of maximal strength.

A useful technique to obtain results for the one dimensional ERW is to employ the well-known relationship between branching processes and random walks. See also \cite{Basdevant08b, Basdevant08a, Zerner08, Bauernschubert12} for this method.
Since there are only cookies of strength 1, we can concentrate on branching processes with immigration and no emigration.
In order to prove Theorem \ref{th:1}, we have to deal with a critical BPIRE. See Section \ref{sec:ERWRE} for the precise connection between our model and the critical BPIRE.
Roughly speaking, an excursion to the right of the random walk can be translated into a branching process by counting the number of up-crossing from $n$ to $n+1$, $n\in\N$, between down-crossings from $n$ to $n-1$.
The translation from the branching process to the excursion is given by the contour process.
The cookies in the ERWRE model correspond to the immigrants and the random environment gives the random offspring distributions for the branching process.
As we will see in Section \ref{sec:ERWRE}, the recurrence of the branching process implies the recurrence of the random walk and vice versa.

Thus, the discussion of BPIRE with focus on its recurrence and transience behavior is essential.

\subsection{Branching process in random environment with immigration}
The literature on branching processes is extensive, see for instance the survey article \cite{Vatutin93}. \cite{Vatutin11} contains a more recent review on branching processes in random environment. Critical branching processes in random environment with immigration are studied e.g.\ by Key in \cite{Key87} and Roitershtein in \cite{Roitershtein07}. Unfortunately, a proper transience and recurrence criteria for our model could not be found or deduced.

Let us introduce the definition of the BPIRE that we study in this article. It differs slightly from the one in \cite[p.\ 344f]{Key87}, see also Remark \ref{remark:Key_def}.

\begin{definition}
	\label{def:BPIRE}
      Consider a sequence $e=(e_n)_{n\in\N}=(r_n,m_n)_{n\in\N}$ of pairs of random variables which take values in the set of probability distributions on $\N_0$. For $n\in\N$, $r_n$ and $m_n$ give the distribution for reproduction, respectively immigration, in generation $n$. Assume that the so-called \emph{random environment} $(e_n)_{n\in\N}$ is i.i.d.\ under some probability measure $Q$ and denote by $Q_e[\cdot]:=Q[\cdot|e]$ the conditional distribution and by $E_e$ the expectation w.r.t.\ $Q_e$.\\
      Furthermore, let $\{\xi_j^{(n)}, M_k;\; j,n,k\in\N\}$ be a family of $\N_0$-valued random variables on the same probability space which is $Q$-a.s.\ independent under $Q_e$ and satisfies $Q$-a.s.\ for $j,n\in\N$
        \begin{align*}
		Q_e[M_n=\cdot]			&=m_n, 		\\
		Q_e[\xi_j^{(n)}=\cdot]	&=r_{n}.
	\end{align*}
      Then, the process \bpire given by $Z_0:=0$ and
      \[
       Z_n:=\xi_1^{(n)}+\ldots+\xi_{Z_{n-1}}^{(n)}+M_n \quad\text{for } n\in\N
      \]
      (or every process with the same distribution) is called \emph{branching process in random environment with immigration} (BPIRE). The random variable $\xi_j^{(n)}$ can be understood as the number of offspring of the $j$-th individual of generation $n-1$ and $M_n$ as the number of immigrants in the $n$-th generation.

      Another useful way to describe the BPIRE is the following.
      For each
      $j\in\N$, let
      $(Z_n(j))_{n\in\N_0}$ be a branching process that starts at time $j$ with $Z_0(j)=M_j$ individuals (or immigrants) and whose reproduction distribution is given by $(r_{n+j})_{n\in\N}$ under $Q_e$.
      More precisely, we consider branching processes that have the same distribution like processes realized by $Z_n(j)=\xi_{j,1}^{(n+j)}+\ldots+\xi_{j,Z_{n-1}(j)}^{(n+j)}$ where, under $Q_e$, $\{\xi_{j,i}^{(k)}, M_n;\; j,i,k,n\in\N\}$ is independent and $\xi_{j,i}^{(k)}$ has distribution $r_{k}$ $Q$-a.s.\\
      Then, the sum over the offspring at the
      same time plus the immigrants at that time,
        \[
         Z_n=\sum_{j=1}^{n}{Z_{n-j}(j)} \quad\text{for } n\in\N,
        \]
      gives a BPIRE.

      The latter definition is similar to Key's definition in \cite[p.\ 344f]{Key87}, see also Remark \ref{remark:Key_def}.

      If $E_Q[\log E_e[\xi_1^{(1)}]]$ exists, \bpire is called \emph{critical, subcritical} or \emph{supercritical} if $E_Q[\log E_e[\xi_1^{(1)}]]=0,<0$ or $>0$ respectively in accordance with the standard classification of branching processes in random environment. For an extended classification see for instance also \cite[p.\ 4]{Vatutin11}.

\end{definition}

Throughout the article, the distribution $r_n$ will be represented by its probability generating function (p.g.f.) denoted by $\varphi_n$. Apart from Lemma \ref{th:Key} and its application in the proof of Theorem \ref{th:recurrence_critBPIRE} it will be furthermore assumed that $m_n$ takes $Q$-a.s.\ values in the set of dirac-measures on $\N_0$, $\{\delta_n, n\in\N_0\}$. In this case let us write $n$ instead of $\delta_{n}$ as a short notation. Thus, for a sequence $(M_n)_{n\in\N}$ of $\N_0$-valued random variables, $(\varphi,M):=(\varphi_n,M_n)_{n\in\N}$ denotes an environment where the distribution for offspring in generation $n$ of an individual in generation $n-1$ is given by the p.g.f.\ $\varphi_n$ and where $M_n$ individuals immigrate in the $n$-th generation $Q_{(\varphi,M)}$-a.s.

\begin{remark}
\label{remark:Key_def}
	In his formulation of the BPIRE-model in \cite{Key87}, Key
	does not count the number of immigrants at time $n$ as a part of generation $n$ but only their offspring as part of the next generation.
\end{remark}

Note that $(Z_n)_{n\geq0}$ is a time-homogeneous Markov chain under $Q$. In this article it is assumed that the BPIRE is irreducible under $Q$ with state space $\N$ or $\N_0$. Motivated by the application to the ERWRE, we are interested in recurrence and transience criteria for a critical BPIRE.

\begin{theorem}
	\label{th:recurrence_critBPIRE}
	Let $(Z_n)_{n\geq0}$ be an irreducible BPIRE with p.g.f.\ $\varphi_n$ for offspring in generation $n$ and $M_n$ immigrants in generation
	$n$. Assume that
	\begin{enumerate}
		\item $(\varphi_n, M_n)_{n\in\N}$ is i.i.d.\ under $Q$.
		\item $E_Q[|\log \mu_1|^2]<\infty$, $E_Q[\log \mu_1]=0$ and $Q[\mu_1= 1]<1$ where $\mu_n:=\varphi_n'(1)$.
		\item $E_Q[(\log_+ M_1)^{2+\epsilon}]<\infty$ for some $\epsilon >0$.
	\end{enumerate}
	Then $(Z_n)_{n\geq 0}$ is recurrent.
\end{theorem}

The next theorem gives a criterion for transience of a critical BPIRE. Let $Q_{\varphi}$ denote the conditional distribution $Q[\cdot|\varphi]$ and write $\Var_{\varphi}$ for the variance according to the measure $Q_{\varphi}$.

\begin{theorem}
  \label{th:transience_critBPIRE}
  Consider an irreducible BPIRE $(Z_n)_{n\geq0}$ with p.g.f.\ $\varphi_n$ for offspring in generation $n$ and $M_n$ immigrants in generation
  $n$. Assume that 
    \begin{enumerate}
      \item $(\varphi_n, M_n)_{n\in\N}$ is i.i.d.\ under $Q$.
      \item $E_Q[|\log   \mu_1|^{\delta}]<\infty$ for every $0<\delta<6$ and $E_Q[\log \mu_1]=0$.
      \item $E_Q[\big(\log_+(\Var_{\varphi}(\xi_1^{(1)})\cdot\mu_1^{-2})\big)^2]<\infty$ where $\mu_n:=\varphi_n'(1)$.
      \item $\liminf_{t\to\infty}(t^{\lambda}\cdot Q[\log{M_1}>t])> 0$ for some $0<\lambda<2$.
    \end{enumerate}
  Then $(Z_n)_{n\geq 0}$ is transient.
\end{theorem}

The recurrence criterion for the branching process was inspired by some work on random difference equations, e.g.\ \cite{Babillot97}, since there is some similarity between these processes. In \cite[p.\ 1196]{Goldie00} Goldie asks about recurrence and transience of random difference equations but characterizes only its positive recurrence.

The present paper is organized as follows.
Section \ref{sec:BPIRE} is dedicated to the proofs of Theorems \ref{th:recurrence_critBPIRE} and \ref{th:transience_critBPIRE} on critical BPIRE.
In Section \ref{sec:ERWRE}, the relation between the ERWRE and the branching process is established in order to prove Theorem \ref{th:1}. Examples are also given for the different cases of the theorem.

\section{Branching process in a random environment with immigration}
\label{sec:BPIRE}

First, let us deduce for our model an analogue result about positive recurrence for an irreducible subcritical BPIRE from Theorem 3.3 in \cite{Key87}.
\begin{lemma}
	\label{th:Key}
	Let $(Z_n)_{n\geq0}$ be a BPIRE with reproduction according to the sequence of p.g.f.\ $(\varphi_n)_{n\in\N}$ and immigration according to probability measures
	$(m_n)_{n\in\N}$. Assume that
	\begin{enumerate}
		\item $(\varphi_n,m_n)_{n\in\N}$ is i.i.d.\ under $Q$.
		\item $E_Q[\log_+ E_{(\varphi_n, m_n)_{n\in\N}}[M_1]]<\infty$.
		\item $E_Q[\log_+ \mu_1]<\infty$ and $E_Q[\log \mu_1]<0$ where $\mu_n:=\varphi_n'(1)$.
	\end{enumerate}
	Then $(Z_n)_{n\geq0}$ is positive recurrent.
\end{lemma}

\begin{proof}
  It is helpful to work with the alternative description of the BPIRE given in Definition \ref{def:BPIRE}.
  Like in \cite{Key87} we amplify this definition in the sense that we do not only consider branching processes $(Z_n(t))_{n\in\N_0}$ starting at positive times, but allow $t\in\Z$. Therefore, the random environment is assumed to be a sequence $e=(\varphi_x,m_x)_{x\in\Z}$ of i.i.d.\ random variables.

  Recall that for $n\geq1$ the BPIRE can be defined as $Z_n=\sum_{j=1}^{n}{Z_{n-j}(j)}$. Key considers in \cite{Key87} in a more general setting BPIRE of the form
    \[
     \tilde{Z}_n^{(1)}:=\sum_{j=1}^{n-1}{Z_{n-j}(j)}.
    \]
  We shift this process and set for $k\in\N_0$,
    \[
     \tilde{Z}_0^{(-k)}:=\sum_{j=1}^{k}{Z_{j}(-j)},
    \]
  which is a BPIRE at time $0$ that started in the past at time $-k$.

  Since $e$ is a sequence of i.i.d.\ random variables and since the branching processes $(Z_n(t))_{n\in\N_0}$, $t\in\Z$, are independent under $Q_e$, we get for $v\in\N_0$ and $n\in\N$,
  \begin{align*}
    Q[Z_n=v]&=Q[Z_0(n)+\tilde{Z}_n^{(1)}=v]\\
    &=\sum_{j=0}^{v}E_Q[Q_e[Z_0(n)=v-j,\;\tilde{Z}_n^{(1)}=j]]\\
    &=\sum_{j=0}^{v}E_Q[Q_e[Z_0(n)=v-j] Q_e[\tilde{Z}_n^{(1)}=j]]\\
    &=\sum_{j=0}^{v}E_Q[Q_e[Z_0(0)=v-j] Q_e[\tilde{Z}_0^{(1-n)}=j]].
  \end{align*}

  According to Lemma 2.2 in \cite{Key87}, $\lim_{n\to\infty}Q_e[\tilde{Z}_0^{(1-n)}=j]$ exists $Q$-a.s.\ for each $j\in\N_0$. Hence, by the dominated convergence theorem, $\pi(v):=\lim_{n\to\infty}Q[Z_n=v]$ exists for every $v\in\N_0$ and

  \begin{equation}
    \label{eq:pi}
    \pi(v)=\sum_{j=0}^{v}E_Q[Q_e[Z_0(0)=v-j] \lim_{n\to\infty}Q_e[\tilde{Z}_0^{(1-n)}=j]].
  \end{equation}

  Let us show now that $\pi$ is a probability measure on $\N_0$. By (\ref{eq:pi})
    \begin{align*}
      \sum_{v\in\N_0}\pi(v)&=\sum_{v\in\N_0}\sum_{j=0}^{v}E_Q[Q_e[Z_0(0)=v-j] \lim_{n\to\infty}Q_e[\tilde{Z}_0^{(1-n)}=j]]\\
      &=\sum_{j\in\N_0}\sum_{v\geq j}E_Q[Q_e[Z_0(0)=v-j] \lim_{n\to\infty}Q_e[\tilde{Z}_0^{(1-n)}=j]]\\
      &=\sum_{j\in\N_0}E_Q[\lim_{n\to\infty}Q_e[\tilde{Z}_0^{(1-n)}=j]].
    \end{align*}

  Note that for all $j\in\N_0$,
    \[
     \tilde{\pi}(j):=E_Q[\lim_{n\to\infty}Q_e[\tilde{Z}_0^{(1-n)}=j]]=\lim_{n\to\infty}E_Q[Q_e[\tilde{Z}_0^{(1-n)}=j]]=\lim_{n\to\infty}Q[\tilde{Z}_n^{(1)}=j]
    \]
  and $\tilde{\pi}$ defines a probability measure on $\N_0$ according to Theorem 3.3 in \cite{Key87}.
  Thus, $\sum_{v\in\N_0}\pi(v)=1$ and the subcritical BPIRE is positive recurrent, see e.g.\ \cite[Theorem 8.18]{Kallenberg02}
\end{proof}

Now we will prove the recurrence and transience criteria for a critical BPIRE.
The recurrence criteria in Theorem \ref{th:recurrence_critBPIRE} is inspired by a similar result for an autoregressive model defined by a random difference equation in the critical case stated in \cite{Babillot97}. Some of the ideas in \cite[p.\ 480f]{Babillot97} will be employed and transferred to our BPIRE-model.

\begin{proof}[Proof of Theorem \ref{th:recurrence_critBPIRE}]
Assume that $Q[M_1 = 0] < 1$ since $Z_n = 0$ $Q$-a.s.\ for all $n\in\N_0$ if $Q[M_1 = 0] = 1$.

Like in \cite{Babillot97}, let us define $Y_0:=0$ and
\begin{align*}
	Y_n&:=\log(\mu_1\cdot\ldots\cdot \mu_n).
\end{align*}
Then $(Y_n)_{n\geq0}$ is an oscillating random walk, i.e.\ $\limsup_{n\to\infty}(\pm Y_n)=\infty$ $Q$-a.s., see e.g.\ \cite[Proposition 9.14]{Kallenberg02}.
The strict descending ladder epochs, see also \cite[XII.1]{Feller22}, are defined by $L_0:=0$ and
\begin{align*}
	L_n&:= \inf\{k>L_{n-1}: Y_k<Y_{L_{n-1}}\}.
\end{align*}
Since $(Y_n)_{n\geq0}$ is oscillating, $L_n$ is $Q$-a.s.\ finite.
Let $L:=L_1$ and note that $E_Q[Y_L]<0$.

Following the strategy in \cite{Babillot97} we consider the subprocess $(Z_{L_n})_{n\geq0}$ and answer the following questions: Is this process a Markov Chain? Is it comparable to \bpire, more precisely, is it some kind and which kind of a branching process? Is it recurrent? The third question is central for the proof of the theorem since the recurrence of the subprocess yields the recurrence of the process itself.\\
Indeed, we show that $(Z_{L_n})_{n\geq0}$ is a subcritical BPIRE.

\begin{figure}[!ht]%
\centering

\scalebox{0.95}{\includegraphics[width=\columnwidth]{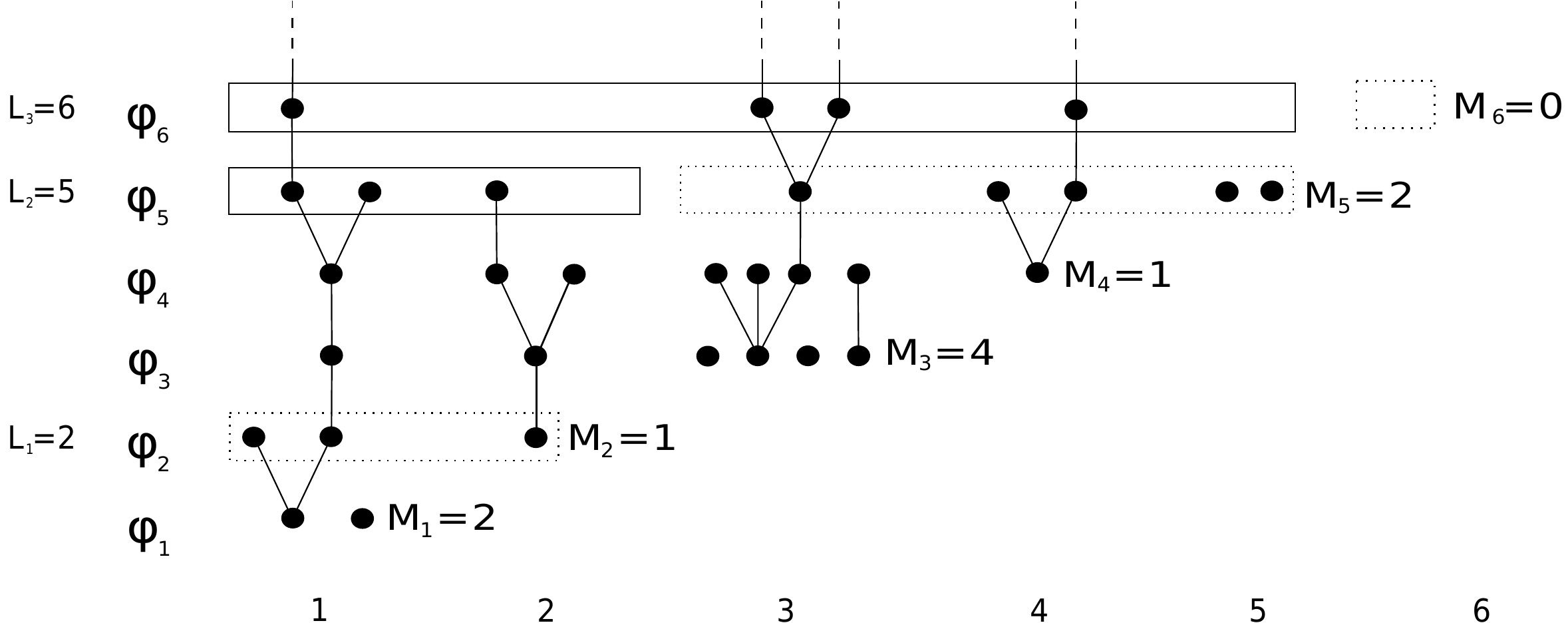}}
\caption{Process and subprocess.}
{\parbox{0.85\textwidth}{\small In this figure it is assumed that $L_1=2$, $L_2=5$ and $L_3=6$. The boxes with continuous outlines mark the offspring of the previous generation, whereas boxes with dotted outline gather the immigrants in generation 1, 2 and 3 of the subprocess $(Z_{L_n})_{n\geq0}$.}}
\label{fig_subprocess}
\end{figure}

A look at Figure \ref{fig_subprocess} or a calculation of $E_{(\varphi,M)}[s^{Z_{L_n}}]$, $0\leq s\leq 1$, $n\in\N$, shows that it is a BPIRE under $Q_{(\varphi,M)}$ with reproduction distribution given by the p.g.f.
\[	\lambda_n(s):=\varphi_{L_{n-1}+1}(\varphi_{L_{n-1}+2}(\ldots(\varphi_{L_n}(s))\ldots)), \quad 0\leq s\leq1,
\]
and measure $m_n$ for immigration in generation $n\in\N$ where $m_n$ is the distribution of
\[	\tilde{M}_n:=
		\sum_{j=L_{n-1}+1}^{L_n}{Z_{L_n-j}(j)}.
\]

The sequence $(\lambda_n, m_n)_{n\in\N}$ is i.i.d.\ under $Q$ since the increments of the ladder epochs $(L_n-L_{n-1})_{n\in\N}$ are i.i.d.\ under $Q$, see \cite[XII.1]{Feller22}.
Subcriticality of $(Z_{L_n})_{n\geq0}$ follows from
$E_Q[\log \lambda'_1(1)]=E_Q[Y_L]<0$.

Furthermore, the following arguments give that the subprocess is still Markovian under $Q$.
For $n\geq1$ and $i_1,\ldots, i_{n+1}\in\N_0$ the Markov property of \bpire under $Q_{(\varphi,M)}$ implies
  \begin{align}
    \label{eq:subprocessMarkovian}
    Q&[Z_{L_{n+1}}=i_{n+1}, Z_{L_n}=i_n, \ldots, Z_{L_1}=i_1]\notag\\
    &=\sum_{k\in\N} E_Q[\mathbbmss{1}_{L_n=k}\cdot Q_{(\varphi,M)}[Z_{L_{n+1}}=i_{n+1}, Z_{L_n}=i_n, \ldots, Z_{L_1}=i_1]]\notag\\
    &=\sum_{k\in\N} E_Q[\mathbbmss{1}_{L_n=k}\cdot Q_{\theta^k(\varphi,M),i_n}[Z_{L_{1}}=i_{n+1}] Q_{(\varphi,M)}[Z_{L_n}=i_n, \ldots, Z_{L_1}=i_1]]
  \end{align}
  where $\theta^k$ denotes the shift $\theta^k(\varphi_j,M_j)_{j\in\N}=(\varphi_{k+j},M_{k+j})_{j\in\N}$ and $i_n$ in $Q_{\theta^k(\varphi,M),i_n}$ denotes the number of ancestors. Since $(\varphi_j,M_j)_{j\in\N}$ is i.i.d.\ it follows that (\ref{eq:subprocessMarkovian}) equals
  \begin{align*}
    &\sum_{k\in\N} E_Q[Q_{\theta^k(\varphi,M),i_n}[Z_{L_{1}}=i_{n+1}]] E_Q[\mathbbmss{1}_{L_n=k}\cdot  Q_{(\varphi,M)}[Z_{L_n}=i_n, \ldots, Z_{L_1}=i_1]]\\
    & = Q_{i_n}[Z_{L_1}=i_{n+1}] Q[Z_{L_n}=i_n, \ldots, Z_{L_1}=i_1].
  \end{align*}

To apply Lemma \ref{th:Key}, we have to check if $E_Q[\log_+ E_{(\varphi, M)}[\tilde{M}_1]]<\infty$ is satisfied.
The integrability of $\log(1+E_{(\varphi, M)}[\tilde{M}_1])$ implies integrability of $\log_+ E_{(\varphi, M)}[\tilde{M}_1]$ and vice versa.
We follow the strategy of the proof of Lemma 5.49 in \cite{Elie82}.
According to \cite[Lemma 5.23]{Elie82}, $\log(1+E_{(\varphi, M)}[\tilde{M}_1])$ is integrable if and only if 
  \begin{equation}
   \label{eq:Elie1}
   \limsup_{n\to\infty}E_{(\varphi, M)}[\tilde{M}_n]^{1/n}<\infty \text{ $Q$-a.s.}
  \end{equation}

Since $E_Q[|\log \mu_1|^2]<\infty$, Theorem 1a in \cite[XII.7, p.\ 414f]{Feller22} can be applied and yields $Q[L>n]\sim c/\sqrt{n}$ for some constant $c>0$. This implies $E_Q[L^{\beta}]<\infty$ for every $0<\beta<1/2$. Let $\beta=\frac{1}{2+\epsilon}$.
E.g.\ by \cite[Theorem 4.23]{Kallenberg02}, we get that $Q$-a.s.\
\begin{equation}
	\limsup_{n\to\infty}\frac{L_n^{\beta}}{n}=0.
	\label{eq:L_n}
\end{equation}

For the proof of (\ref{eq:Elie1})
note that by the definitions of $Z_{k}(j)$ and since $\mu_{j}\cdot\ldots\cdot\mu_{L_n}<1$ for each $j<L_n$ we obtain $Q$-a.s.\
\begin{align}
  \label{eq:Elie2}
 	E_{(\varphi, M)}[\tilde{M}_n]
	&=\sum_{j=L_{n-1}+1}^{L_n}E_{(\varphi, M)}[Z_{L_n-j}(j)]\notag\\
	&=\sum_{j=L_{n-1}+1}^{L_n} M_j\mu_{j+1}\cdot\ldots\cdot\mu_{L_n}\notag\\
	&\leq \sum_{j=1}^{L_n}M_j.
\end{align}

Now, an analog calculation as the one in \cite[p.\ 336]{Elie82} yields (\ref{eq:Elie1}). For completeness let us give the full argument.
Inequality (\ref{eq:Elie2}) gives
\begin{align}
  \label{eq:Elie4}
	\limsup_{n\to\infty} E_{(\varphi, M)}[\tilde{M}_n]^{1/n}
		&\leq \limsup_{n\to\infty} \exp\left(\frac{1}{n}\log\bigg(\sum_{j=1}^{L_n}{M_j}\bigg)\right)\notag\\
		&\leq \limsup_{n\to\infty}\exp\left(\frac{1}{L_n^{\beta}}\log\bigg(1+\sum_{j=1}^{L_n}M_j\bigg)\cdot \frac{L_n^{\beta}}{n}\right).
\end{align}
Since
  \[	\log\left(1+\sum_{j=1}^{L_n}{M_j}\right)
			\leq \sup_{j=1}^{L_n}{\log(1+ M_j)} + \log L_n
  \]
and $\log L_n/L_n^{\beta}\to 0$ for $n\to\infty$, we get

  \begin{align}
    \label{eq:Elie3}
    \limsup_{n\to\infty} \frac{1}{L_n^{\beta}} \log\left(1+\sum_{j=1}^{L_n}{M_j}\right)
    & \leq\limsup_{n\to\infty}\frac{1}{L_n^{\beta}}\left(\sup_{j=1}^{L_n}{(\log(1+ M_j))^{\frac{1}{\beta}}}\right)^{\beta}\notag\\
    & \leq\limsup_{n\to\infty}\left(\frac{\sum_{j=1}^{L_n}{(\log(1+ M_j))^{\frac{1}{\beta}}}}{L_n}\right)^{\beta}.
  \end{align}

Recall that $E_Q[(\log_+ M_1)^{\frac{1}{\beta}}]=E_Q[(\log_+ M_1)^{2+\epsilon}]<\infty$. Hence, the right-hand side of (\ref{eq:Elie3}) is finite $Q$-a.s.\ by the law of large numbers and (\ref{eq:Elie1}) follows by (\ref{eq:L_n}), (\ref{eq:Elie4}) and (\ref{eq:Elie3}).

Summing up, the subprocess $(Z_{L_n})_{n\geq0}$ is a BPIRE with reproduction according to $(\lambda_n)_{n\in\N}$ and immigration distribution $(m_n)_{n\in\N}$ where $(\lambda_n,m_n)_{n\in\N}$
is i.i.d.\ and $E[\log E_{(\varphi, M)}[\tilde{M}_1]]<\infty$.
Applying Lemma \ref{th:Key} gives that $(Z_{L_n})_{n\geq0}$ is positive recurrent, in particular recurrent. Hence, \bpire is recurrent.
\end{proof}

Before proving Theorem \ref{th:transience_critBPIRE} let us deduce a useful result from Theorem 2 in \cite{Zerner02}.

\begin{lemma}
  \label{lem:Hilfslemma1}
  Let $d\in\N$, $c>0$ and $0<a<1$. Assume that $\{V, V_{i,n}; i,n\in\N_0\}$ is a family of i.i.d., almost surely nonnegative random variables. Then,
  $E_Q[(\log_+V)^d]<\infty$ if and only if $\sum_{n\in\N_0}a^n\sum_{i=0}^{\lfloor c n^{d-1}\rfloor}V_{i,n}<\infty$ $Q$-a.s.\
\end{lemma}

\begin{proof}
  The Lemma is shown by induction over $d$.
  Lemma 4 in \cite{Zerner02} gives that $E_Q[\log_+V]<\infty$ if and only if $E_Q[\log_+(\sum_{i=0}^{\lfloor c \rfloor}V_{i,0})]<\infty$. Furthermore, the latter is equivalent to the almost sure convergence of $\sum_{n\in\N_0}a^n\sum_{i=0}^{\lfloor c\rfloor}V_{i,n}$, see for instance \cite[Theorem 5.4.1]{Lukacs75}.

  Let the assertion hold for $d\geq 1$. It will be useful to consider random variables with three indices. Therefore, let $\{V, V_{j,i,n}; j,i,n\in\N_0\}$ be i.i.d. Theorem 2 in \cite{Zerner02} yields that $E_Q[(\log_+V)^{d+1}]$ is finite if and only if $E_Q[(\log_+V_0)^{d}]$ is finite, where $V_0=\sum_{n\in\N_0}a^n V_{0,0,n}$.
  By induction hypothesis, $E_Q[(\log_+V_0)^{d}]<\infty$ is equivalent to the almost sure convergence of
    \[
     \sum_{j\in\N_0}a^j\sum_{i=0}^{\lfloor c j^{d-1} \rfloor}\sum_{n\in\N_0}a^n V_{j,i,n}=\sum_{j\in\N_0}\sum_{n\in\N_0}a^{j+n}\sum_{i=0}^{\lfloor c j^{d-1}\rfloor} V_{j,i,n}=\sum_{k\in\N_0}a^{k}\sum_{j=0}^{k}\sum_{i=0}^{\lfloor c j^{d-1}\rfloor} V_{j,i,k-j}.
    \]

  The number of summands in $\sum_{j=0}^{k}\sum_{i=0}^{\lfloor c j^{d-1}\rfloor} V_{j,i,k-j}$ is asymptotically equal to $k^d c/d$ for $k\to\infty$. For $d\geq1$ note that the almost sure convergence of $\sum_{n\in\N_0}a^n\sum_{i=0}^{\lfloor\tilde{c} n^{d}\rfloor}V_{i,n}$ for some $\tilde{c}>0$ and all $0<a<1$ implies the almost sure convergence for all $\tilde{c}>0$. Hence, the lemma follows.
\end{proof}

Let us now prove the transience criterion for an irreducible critical BPIRE.

\begin{proof}[Proof of Theorem \ref{th:transience_critBPIRE}]
  The strategy of the proof is similar to the one of Theorem 2.2 in \cite{Bauernschubert12}. First we discuss an autoregressive model defined by the critical random
  difference equation $X_n:=\mu_n X_{n-1}+M_n$ for $n\in\N$ and $X_0=0$. We will show that $Q$-a.s.
	\begin{equation}
	 X_n> e^{\sqrt{n}} \text{ for $n$ large.}
	\label{eq:liminf}
	\end{equation}
  Thereafter, \arp is coupled with the critical BPIRE \bpire to obtain the transience.
  
  Like in the proof of Theorem \ref{th:recurrence_critBPIRE} we define $Y_n=\log\mu_1+\ldots+\log\mu_n$. Let $\frac{1}{2}<\kappa<\frac{1}{\lambda}$ and
  \[ T(\kappa):=\inf\{k\in\N: Y_n\geq -n^{\kappa} \text{ for all } n\geq k\}.
  \]
  By the law of the iterated logarithm $T(\kappa)$ is finite $Q$-a.s. Choose $0<\delta<6$ such that $\gamma:=\delta\kappa-2>1$. Then
  \begin{align*}
    E_Q[T(\kappa)^{\gamma}] &= \sum_{n\in\N} n^{\gamma} Q[T(\kappa)=n]\\
    &\leq 1+\sum_{n\geq 2} n^{\gamma} Q[Y_{n-1}< -(n-1)^{\kappa}]\\
    &\leq 1+\sum_{n\geq 2} n^{\gamma} Q[|Y_{n-1}|> (n-1)^{\kappa}].
  \end{align*}

  This is finite due to the complete convergence theorem of Baum and Katz in \cite[Theorem 3, (a) $\Rightarrow$ (b)]{Baum65} and the assumptions $E_Q[|\log\mu_1|^{\delta}]<\infty$
  and $E_Q[\log \mu_1]=0$. Therefore,
    \begin{equation}
      E_Q[T(\kappa)^{\gamma}]<\infty.
      \label{eq:finiteEofT}
    \end{equation}

  Consider now the autoregressive model. The recursion of $X_n$ yields
    \[ X_n= M_n+\mu_n M_{n-1}+\mu_n\mu_{n-1}M_{n-2}+\ldots+\mu_2\ldots\mu_n M_1.
    \]
    
  Set
    \[ W_n:=M_1+\mu_1 M_2+\mu_1\mu_2 M_3+\ldots+\mu_1\ldots\mu_{n-1} M_n
    \]
  for $n\in\N$. Then, exchangeability implies for all $n\in\N$,
    \begin{equation}
    Q[X_n> e^{\sqrt{n}}]=Q[W_n> e^{\sqrt{n}}].
\label{eq:exchangeability}
\end{equation}

   Recall that $\frac{1}{2}<\kappa<\frac{1}{\lambda}$ and $\gamma>1$. By the assumption of the theorem there is some constant $c_1>0$ such that $Q\left[M_{1}> e^{2n^{\kappa}}\right]> c_1\cdot n^{-\kappa\lambda}$ for large $n$. Thus, we get for a suitable constant $c_2>0$ the following bound from above for large $n\in\N$.
  \begin{align*}
    \label{eq:boundW_n1}
    Q[W_n\leq e^{\sqrt{n}}, T(\kappa)< n^{\frac{1}{\gamma}}] 
     &\leq Q\Bigg[\bigcap_{n^{\frac{1}{\gamma}}\leq i < n}\left\{\mu_1\ldots\mu_i M_{i+1}\leq e^{\sqrt{n}}\right\},T(\kappa)<  
     n^{\frac{1}{\gamma}} \Bigg]\notag\\
     &= Q\Bigg[\bigcap\limits_{n^{\frac{1}{\gamma}}\leq i < n}\left\{M_{i+1}\leq e^{\sqrt{n}-Y_i}\right\}, T(\kappa)<n^{\frac{1}{\gamma}} \Bigg]\notag\\
     &\leq Q\Bigg[\bigcap\limits_{n^{\frac{1}{\gamma}}\leq i < n}\left\{M_{i+1}\leq e^{2n^{\kappa}}\right\} \Bigg]\notag\\
     &\leq \left(1-Q\left[M_{1}> e^{2n^{\kappa}}\right]\right)^{n-n^{\frac{1}{\gamma}}-1}\notag \\
     &\leq e^{- c_2 n^{1-\kappa\lambda}}.
  \end{align*}
  
  Hence
  \begin{align*}
    Q[W_n\leq e^{\sqrt{n}}]&\leq Q[W_n\leq e^{\sqrt{n}}, T(\kappa)< n^{\frac{1}{\gamma}}] +Q[T(\kappa)\geq n^{\frac{1}{\gamma}}]\\
     &\leq e^{- c_2 n^{1-\kappa\lambda}} + Q[T(\kappa)\geq n^{\frac{1}{\gamma}}]
  \end{align*}
  for large $n$. By (\ref{eq:exchangeability}), (\ref{eq:finiteEofT}) and since $\kappa\lambda <1$, we obtain
  \[ \sum\limits_{n\geq1}Q\left[X_n\leq e^{\sqrt{n}}\right]<\infty
  \]
  and (\ref{eq:liminf}) follows by Borel-Cantelli.
  
  The next step is to couple \arp and $(Z_n)_{n\geq0}$. As can be seen by the notations, the increments of the difference equation correspond to the number of immigrants in the BPIRE and the
  multiplication factor $\mu_n$ to the expected number of offspring of an individual in generation $n-1$.

  Result (\ref{eq:liminf}) implies $Q\left[\bigcap_{n\geq 1}\{X_n>e^{\sqrt{n}}\}\right]>0$ and hence $Q_{\varphi}\left[\bigcap_{n\geq 1}\{X_n>e^{\sqrt{n}}\}\right]>0$ on some $Q$-non-null set $D$. Fix $e^{-1}<\beta<1$. We will show that
  \begin{equation}
    Q_{\varphi}\left[\bigcap\limits_{n\in\N}\{Z_n\geq \beta^{\sqrt{n}} X_n\}\Bigg| \bigcap\limits_{k\in\N}\{X_k> e^{\sqrt{k}}\}
    \right]>0
    \label{eq:Zn>betaXn}
  \end{equation}
  on $D$.
  Hence,
  \[ Q\left[\lim\limits_{n\to\infty}Z_n=\infty\right]>0
  \]
  since $e\beta>1$, and the transience follows.
  
  For $n\in\N_0$ let
  \[ B_n:=\bigcap_{j=1}^{n}\{Z_j\geq \beta^{\sqrt{j}} X_j\}\cap\bigcap_{k\geq 1}\{X_k >e^{\sqrt{k}}\}.
  \]
  
  By the definition of \bpire in Definition \ref{def:BPIRE} we get on $D$
  \begin{align*}
    &Q_{\varphi}\left[Z_n <\beta^{\sqrt{n}} X_n, B_{n-1}\right]
    =\sum\limits_{k\in\N}Q_{\varphi}\left[ Z_n <\beta^{\sqrt{n}} X_n,Z_{n-1}=k, B_{n-1}\right] \notag\\
    &=\sum\limits_{k\in\N}Q_{\varphi}\left[ \mu_n k-\sum\limits_{i=1}^{k}{\xi_i^{(n)}}> \mu_n
    (k-\beta^{\sqrt{n}} X_{n-1})+ M_n(1-\beta^{\sqrt{n}}),Z_{n-1}=k, B_{n-1}\right] \notag\\
    &\leq\sum\limits_{k>(e\beta)^{\sqrt{n-1}}}Q_{\varphi}\left[ \mu_n k-\sum\limits_{i=1}^{k}{\xi_i^{(n)}} 
    >(1-\beta^{\sqrt{n}-\sqrt{n-1}})\mu_n k\right]\cdot Q_{\varphi}\left[Z_{n-1}=k, B_{n-1}\right].
  \end{align*}
  Chebyshev's inequality implies
  \begin{align*}
    Q_{\varphi}\left[ \mu_n k-\sum\limits_{i=1}^{k}{\xi_i^{(n)}}>(1-\beta^{\sqrt{n}-\sqrt{n-1}})\mu_n k\right]
    &\leq  \frac{\Var_{\varphi}\left({\xi_1^{(n)}}\right)}{(1-\beta^{\sqrt{n}-\sqrt{n-1}})^2\mu_n^2 k}\; .
  \end{align*}

  Note that for large $n$, $1-\beta^{\sqrt{n}-\sqrt{n-1}}\geq -\frac{1}{2}(\sqrt{n}-\sqrt{n-1})\log\beta$. Hence,
  \begin{align*}
    Q_{\varphi}\left[Z_n <\beta^{\sqrt{n}} X_n, B_{n-1}\right] \leq \frac{4\Var_{\varphi}\left({\xi_1^{(n)}}\right)}{ \mu_n^2
    (\log\beta)^2 (\sqrt{n}-\sqrt{n-1})^2(e\beta)^{\sqrt{n-1}}}\cdot Q_{\varphi}\left[B_{n-1}\right]
  \end{align*}
  for large $n$. Thus, it holds for some $0<\alpha<1$ that $Q_{\varphi}\left[Z_n <\beta^{\sqrt{n}} X_n | B_{n-1}\right]\leq \alpha^{\sqrt{n-1}}\Var_{\varphi}(\xi_1^{(n)})\mu_n^{-2}$ for large $n$ and
    \begin{align*}
      \sum_{n\in\N_0}\alpha^{\sqrt{n}}\Var_{\varphi}(\xi_1^{(n+1)})\mu_{n+1}^{-2}
       &=\sum_{k\in\N_0}\sum_{n=k^2}^{(k+1)^2-1}\alpha^{\sqrt{n}}\Var_{\varphi}(\xi_1^{(n+1)})\mu_{n+1}^{-2} \\
       &\leq \sum_{k\in\N_0}\alpha^{k}\sum_{n=k^2}^{(k+1)^2-1}\Var_{\varphi}(\xi_1^{(n+1)})\mu_{n+1}^{-2}.
    \end{align*}

  Applying assumption $E_Q[\big(\log_+(\Var_{\varphi}(\xi_1^{(1)})\cdot\mu_1^{-2})\big)^2]<\infty$, Lemma \ref{lem:Hilfslemma1} yields on $D$
  \begin{align}
    \label{eq:sum_Qn}
    \sum_{n\in\N}&Q_{\varphi}\left[Z_n <\beta^{\sqrt{n}} X_n \Big| B_{n-1}\right]<\infty.
  \end{align}

  Furthermore, we have on $D$
  \begin{align}
    Q_{\varphi}\left[Z_n \geq \beta^{\sqrt{n}} X_n \Big| B_{n-1}\right]\geq Q_{\varphi}\left[\sum\limits_{i=1}^{Z_{n-1}}{\xi_i^{(n)}} \geq \mu_n \beta^{\sqrt{n}-\sqrt{n-1}} Z_{n-1} \Big| B_{n-1}\right]>0
    \label{eq:eachQn>0}
  \end{align}
  for all $n\geq 1$.
  The left-hand side of (\ref{eq:Zn>betaXn}) equals $\prod_{n\in\N}Q_{\varphi}\left[Z_n \geq \beta^{\sqrt{n}} X_n | B_{n-1}\right]$
  and thus, (\ref{eq:Zn>betaXn}) follows from (\ref{eq:sum_Qn}) and (\ref{eq:eachQn>0}).
\end{proof}

\begin{remark}
  The recurrence and transience criteria in Theorems \ref{th:recurrence_critBPIRE} and \ref{th:transience_critBPIRE} hold in the same way for a BPIRE with one ancestor. Starting with $Z_0=0$ in Definition \ref{def:BPIRE} is only due to the proof of Lemma \ref{th:Key}.
  To make sure that the BPIRE dies out infinitely often $Q$-a.s.\ in Theorem \ref{th:recurrence_critBPIRE}, we have to assume additionally --- if the process starts with one ancestor --- that $Q[\varphi_1(0)>0, M_1=0]>0$.
\end{remark}

\section{Excited random walk in random environment}
\label{sec:ERWRE}

The aim of this section is to transfer the recurrence and transience criteria from the BPIRE to the ERWRE.
Therefore, note the well-known connection between branching processes with migration and excited random walks. For a simple symmetric random walk disturbed by cookies, this idea was employed for instance in \cite{Basdevant08b, Basdevant08a, Zerner08}. In \cite{Bauernschubert12} the author explains the connection between a left-transient RWRE disturbed by cookies of maximal strength and a subcritical BPIRE.
In this section we establish an analogous relation between a critical BPIRE and a recurrent RWRE disturbed by cookies of maximal strength. The purpose of this connection is to prove Theorem \ref{th:1}.

Let us introduce the following notation and variables to describe the connection. For detailed explanations we refer to \cite{Zerner08} or \cite{Bauernschubert12}.

Let $X_i^{(j)}$, $i\in\N$, $j\in\Z$, be a family of independent $\pm1$-valued random variables on $\Omega'$,
such that
  \[P_{z,\omega}[X_i^{(j)}=1]=\omega(j,i)\text{  and  }P_{z,\omega}[X_i^{(j)}=-1]=1-\omega(j,i).\]
Then the ERWRE can be realized recursively by
  \[S_{n+1}=S_n + X_{\#\{m\leq n:\: S_m=S_n\}}^{(S_n)} \qquad\text{ for } n\geq0.\]

The events $\{X_i^{(j)}=1\}$ and $\{X_i^{(j)}=-1\}$ are called success and failure respectively. Furthermore, set
  \begin{align*}
    \xi_j^{(k)}&:= \# \{\text{successes in $\big(X_i^{(k)}\big)_{i>M_k}$ between the $(j-1)$-st and the $j$-th failure}\}\; ,\\
    V_0&:=1\; ,\\
    V_k&:=\xi_1^{(k)}+\ldots+\xi_{V_{k-1}}^{(k)}+M_k\; .
  \end{align*}

Under Assumption \ref{as_A},
it is obtained
that $(V_k)_{k\geq 0}$ is a BPIRE under $P_1$, with one ancestor, immigrants $(M_k)_{k\geq1}$ and progeny given by $(\xi_i^{(j)})_{i,j\in\N}$. Remark that $\xi_i^{(j)}$ has geometric distribution with parameter $(1-p_j)$, ($geo_{\N_0}(1-p_j)$ for short), i.e.\ $P_{1,\omega}[\xi_i^{(j)}=n]=p_j(\omega)^n\cdot (1-p_j(\omega))$ for $n\in\N_0$ and $\mathbb{P}$-a.e.\ $\omega\in\Omega$.

The time when the ERWRE first hits $k\in\Z$, is denoted by
  \[T_k:=\inf\{n\in\N: S_n=k\}.\]

The recurrence from the right of the ERWRE is defined analogously to \cite[p.\ 1962]{Zerner08} or \cite{Bauernschubert12}.

\begin{definition}
\label{def:rekurrentvonrechts}
  The ERWRE is called \emph{recurrent from the right}, if the first excursion to the right of $0$, if there is any, is $P_0$-a.s.\ finite, i.e.\ $P_1[T_0 <\infty]=1$.
\end{definition}

  Note that, by the assumption of a recurrent RWRE and cookies inducing a drift to the right, it follows that
    \begin{equation}
      \label{eq:limsup}
      P_0[\limsup_{n\to\infty}S_n=+\infty]=1.
    \end{equation}
  Indeed, for $t>0$ and $k\in\N$, monotonicity of $P_{0,\omega}[T_k \leq t]$ with respect to the environment holds according to Lemma 15 in \cite{Zerner05}, which can be extended from $\omega \in ([1/2,1]^{\N})^{\Z}$ to the current situation $\omega\in\Omega=([0,1]^{\N})^{\Z}$.

  Using ($\ref{eq:limsup}$) and $P_0[\liminf_{n\to\infty}S_n\in\{\pm\infty\}]=1$ instead of Lemma 3.2 in \cite{Bauernschubert12} we can prove the next result about the connection between ERWRE and BPIRE analogously to Lemma 3.5, Lemma 3.6 and Lemma 3.7 in \cite{Bauernschubert12}.

  \begin{lemma}
    \label{lem:recurrence}
    Let Assumption \ref{as_A} hold.
    The ERWRE $(S_n)_{n\geq0}$ is recurrent from the right if and only if $(V_k)_{k\geq0}$ is recurrent in $0$, i.e.\ $P_1[\exists k\in\N: V_k=0]=1$.\\
    If $(S_n)_{n\geq0}$ is recurrent from the right, then all excursions are $P_0$-a.s.\ finite.\\
    If $(S_n)_{n\geq0}$ is not recurrent from the right, then $P_0\left[\lim_{n\to\infty}S_n=+\infty\right]>0$.
  \end{lemma}

  \begin{remark}
    In the case of right-recurrence note that --- due to monotonicity --- there are, in contrary to the model in \cite{Bauernschubert12}, a.s.\ infinitely many finite excursions to the right since the underlying random environment induces a recurrent random walk. Hence, $P_0\left[S_n=0 \text{ infinitely often}\right]=1$ if $(S_n)_{n\geq 0}$ is recurrent from the right.
  \label{remark:excursions}
  \end{remark}

\begin{proof}[Proof of Theorem \ref{th:1}]
  The process $(V_k)_{k\geq 0}$ as described above is a BPIRE with immigrants $(M_n)_{n\geq1}$ and offspring distribution
  $geo_{\N_0}(1-p_j)$, $j\in\N$.
  It is irreducible on the state space $\N_0$ since $0<p_0<1$ $\mathbb{P}$-a.s.\ and $\mathbb{P}[M_0=0]>0$ by Assumption \ref{as_A}. Given an environment $\omega\in\Omega$, the expected number of offspring produced by a single particle in the $(j-1)$-st generation and its variance are
    \[
      \mu_j(\omega):=E_{0,\omega}[\xi_1^{(j)}]=\frac{p_j(\omega)}{1-p_j(\omega)}=\rho_j^{-1}(\omega)
    \]
  and
    \[
      \Var_{0,\omega}(\xi_1^{(j)})=\frac{p_j(\omega)}{(1-p_j(\omega))^2}, \text{ respectively}.
    \]

  Hence, since \ref{as_A} is assumed for Theorem \ref{th:1}, $(V_k)_{k\geq 0}$ is a critical BPIRE according to Definition \ref{def:BPIRE}.
  Furthermore, $\mathbb{P}[\mu_1=1]<1$ holds by Assumption \ref{as_A}. Supposing that $\mathbb{E}[|\log\mu_1|^{\delta}]<\infty$ for every $0<\delta<6$ also includes in particular that $\mathbb{E}[(\log p_1)^2]<\infty$.
  To see this, note that
    \[
      \mathbb{E}\left[\Big(\log \frac{p_1}{1-p_1}\Big)^2\right]\geq \mathbb{E}\left[\Big(\log \frac{p_1}{1-p_1}\Big)^2\mathbbmss{1}_{p_1<\kappa}\right]
    \]
  for any $\kappa>0$. Now, for $\kappa$ small enough, a constant $c>0$ can be found such that
    \[
     \Big(\log \frac{p_1}{1-p_1}\Big)^2=(\log p_1 - \log (1-p_1))^2\geq c\cdot(\log p_1)^2
    \]
  holds on $\{0<p_1<\kappa\}$. 

  Therefore, $\mathbb{E}[\big(\log_+ (\Var_{0,\omega}(\xi_1^{(1)})\cdot \mu_1^{-2})\big)^2]=\mathbb{E}[(\log p_1)^2]<\infty$ is fulfilled and finally, the assumptions of Theorems \ref{th:recurrence_critBPIRE} and \ref{th:transience_critBPIRE} are satisfied.

  Thus, if $\mathbb{E}[(\log_+ M_0)^{2+\epsilon}]<\infty$ holds for some $\epsilon>0$, then the BPIRE $(V_k)_{k\geq 0}$ is recurrent by Theorem \ref{th:recurrence_critBPIRE}. Lemma \ref{lem:recurrence} gives that $P_0\left[S_n=0 \text{ infinitely often}\right]=1$ since the underlying RWRE is recurrent and the first statement of Theorem \ref{th:1} follows.

  Let $\liminf_{t\to\infty}(t^{\lambda}\cdot \mathbb{P}[\log{M_1}>t])> 0$ for some $0<\lambda<2$. Then $(V_k)_{k\geq 0}$ is transient by Theorem \ref{th:transience_critBPIRE}. We get thus by Lemma \ref{lem:recurrence} that $P_0\left[\lim_{n\to\infty}S_n=+\infty\right]>0$.
  Now, following the same strategy as in the proof of Theorem 1.1.(iii) in \cite{Bauernschubert12} yields
  $P_0[\lim_{n\to\infty}S_n=+\infty]=1$. Hence, the proof is complete.
\end{proof}

\begin{example}
 Suppose that the assumptions of Theorem \ref{th:1} are fulfilled and let $\lambda >0$. Let $M_0$ satisfy
  \begin{align*}
   \mathbb{P}[M_0\geq k]&=\frac{1}{(1+\log k)^{\lambda}}\quad \text{ for } k\geq2,\; k\in\N, \\
   \mathbb{P}[M_0=1]&=0,\\
   \mathbb{P}[M_0=0]&=1-\frac{1}{(1+\log 2)^{\lambda}}.
  \end{align*}
 Theorem \ref{th:1} makes no statement on the case $\lambda=2$, but for $\lambda\neq2$ we obtain the following results.\\
 If $\lambda<2$, then $\lim_{n\to\infty}S_n= +\infty$ $P_0$-a.s.\ due to $\lim_{t\to\infty} t^{\lambda}\mathbb{P}[\log M_0\geq t]=1>0$.
 If $\lambda>2$, then $S_n=0$ infinitely often $P_0$-a.s.\ since we can choose $\epsilon>0$ such that $2+\epsilon <\lambda$ and get $\mathbb{E}[(\log_+ M_0)^{2+\epsilon}]<\infty$.
\end{example}

\section*{Acknowledgements}
This work was funded by ERC Starting Grant 208417-NCIRW.
Many thanks go to my advisor Prof.\ Martin P.\ W.\ Zerner for helpful suggestions and his support. I thank Elmar Teufl
for inspiring discussions.

\bibliographystyle{amsplain}\vspace*{-8mm} 
\bibliography{bibfile}
\parindent=0pt 
\end{document}